\documentclass[11pt]{amsart}
\usepackage{amsfonts,amscd}
\usepackage[utf8]{inputenc}
\usepackage{hyperref}
\usepackage{amssymb}
\usepackage{url}
\usepackage{graphicx}
\usepackage[english]{babel}
\theoremstyle{plain}
\newtheorem{theorem}                {Theorem}      [section]
\newtheorem{proposition}  [theorem]  {Proposition}
\newtheorem{corollary}    [theorem]  {Corollary}

\theoremstyle{definition}

\newtheorem{remark}       [theorem]  {Remark}

\newcommand{\sech}{\operatorname{sech}}
\numberwithin{equation}{section}

\def \R{{\mathbb R}}
\def \s{{\mathbb S}}

\def \n{{\mathbb N}}
\def \h{{\mathbb H}}

\def \Hh{{\mathbb H}}
\newcommand{\Ric}{\operatorname{Ric}}

\newcommand{\Scal}{\operatorname{Scal}}

\usepackage{color}

\def \1 {\`}

\DeclareMathOperator{\grad}{grad}
\DeclareMathOperator{\Div}{div}

\numberwithin{equation}{section}

\title[Biconservative and Biharmonic Hypersurfaces]{Biconservative and Biharmonic Hypersurfaces in Product Spaces}
\author{V.~Branding}
\address{ Institute of Mathematics\\
University of Rostock\\
Ulmenstraße 69, 18057 Rostock, Germany}
\email{volker.branding@uni-rostock.de}

\author{S.~Montaldo}
\address{Universit\`a degli Studi di Cagliari\\
Dipartimento di Matematica e Informatica\\
Via Ospedale 72\\
09124 Cagliari, Italia}
\email{montaldo@unica.it}

\author{C.~Oniciuc}
\address{Faculty of Mathematics\\ ``Al.I. Cuza'' University of Iasi\\
Bd. Carol I no. 11 \\
700506 Iasi, ROMANIA}
\email{oniciucc@uaic.ro}

\author{A.~Ratto}
\address{Universit\`a degli Studi di Cagliari\\
Dipartimento di Matematica e Informatica\\
Via Ospedale 72\\
09124 Cagliari, Italia}
\email{rattoa@unica.it}

\usepackage{hyperref}
\begin{document}
\begin{abstract} In this paper we investigate biconservative and biharmonic hypersurfaces in product spaces $L^m(\varepsilon)\times 
\R$, where $L^m(\varepsilon)$ denotes the space form $\s^m$ or $\h^m$ of
constant sectional curvature $\varepsilon=1$ or $\varepsilon=-1$ respectively. We obtain their complete description in the case of isoparametric hypersurfaces and hypersurfaces of constant angle. Moreover, we prove some classification results for biconservative and biharmonic hypersurfaces for which $\grad H$ is a principal direction.  In this context  we obtain new examples of {\em complete}, {\em non-cylindrical}, {\em non-totally umbilical} biconservative hypersurfaces in $\h^m\times \R$. Then we prove a rigidity theorem related to these examples. Finally, we prove a global non-existence theorem for complete proper biharmonic hypersurfaces in $\h^m \times \R$ such that $H^2$ is bounded and the Ricci curvature is bounded from below.
\end{abstract}



\thanks{The authors S.M. and A.R. are members of the Italian National Group G.N.S.A.G.A. of INdAM and their work is partially supported by the Project {PRoBIKI} funded by Fondazione di Sardegna.}

\maketitle

\section{Introduction and statement of the results} 
 
In this paper we study biconservative and biharmonic hypersurfaces in product spaces $L^m(\varepsilon)\times 
\R$, where $L^m(\varepsilon)$ denotes the space form $\s^m$ or $\h^m$ of
constant sectional curvature $\varepsilon=1$ or $\varepsilon=-1$ respectively. A \textit{biharmonic hypersurface} is a Riemannian immersion $\phi: M^m \hookrightarrow L^m(\varepsilon)\times 
\R$ which is a critical point of the \textit{bienergy functional}, that is, such that its \textit{bitension field} $\tau_2(\phi)$ vanishes. A \textit{biconservative hypersurface} is defined by the vanishing of the tangential component of $\tau_2(\phi)$. Note that \textit{minimal immersions} are trivially biconservative and biharmonic. Therefore, in this paper we say that a biconservative or biharmonic hypersurface is \textit{proper} if it is non-minimal. We refer to \cite{MR4410183,MR4265170} for background on biharmonic and biconservative hypersurfaces.

A first investigation of proper biharmonic hypersurfaces $\phi:M^{m}\hookrightarrow L^m(\varepsilon)\times \mathbb{R}$ was carried out in the work \cite{MR4333889} of Maeta and Ou. To describe their main results it is convenient to introduce the \textit{angle function} $\alpha$. More precisely, let $\partial_t$ be the unit global vector field tangent to the second factor of the
ambient space $L^m(\varepsilon)\times \mathbb R$. We assume throughout the paper that $M^m$ is oriented and we decompose $\partial_t$ as follows:
\begin{equation}
\label{eq:dt}
\partial_t=T+\cos\alpha\,\eta
\end{equation}
where $\cos\alpha= \langle \partial_t, \eta\rangle$. Here $\eta$ denotes the unit normal vector field of the hypersurface, while $T$ is the component of $\partial_t$ along the tangent space to $M^m$. Now, the bitension field $\tau_2(\phi)$ can be described by:
\begin{proposition}[\cite{MR4333889}] \label{Prop-Ou} Let $\phi:M^{m}\hookrightarrow L^m(\varepsilon)\times \mathbb{R}$ be a hypersurface and denote by $H$ its mean curvature function. Then the bitension field $\tau_2(\phi)=\tau_2$ is given by
$$
\tau_2=\left[\tau_2\right]^{\top}+\left[\tau_2\right]^{\perp}
$$
where
\begin{equation}\label{BHP1}
\begin{aligned}
\left[\tau_2\right]^{\top}= & -2 m A( \grad H)-m^2 H \grad H-2 m(m-1) \varepsilon H \cos \alpha\, T 
\end{aligned}
\end{equation}
is the tangential component, while
\begin{equation}\label{BHP2}
\begin{aligned}
\left[\tau_2\right]^{\perp} = &m \big\{-\Delta H-  H |A|^2 + \varepsilon  (m-1) H \sin^2\alpha\big\}\eta 
\end{aligned}
\end{equation}
is the normal component.
\end{proposition}

It is well known that constant mean curvature (CMC) hypersurfaces in space forms are biconservative. This is not always true in $L^m(\varepsilon)\times\R$ and our first simple result is:
\begin{proposition}\label{Prop-CMC} Let $M^m$ be a {non-minimal} CMC hypersurface in $L^m(\varepsilon)\times\R$.
\begin{itemize}
\item[{\rm (i)}] If $M^m$ is biconservative, then it is a cylinder $M^{m-1}\times\R$, where $M^{m-1}$ is a non-minimal CMC hypersurface in $L^m(\varepsilon)$.
\item[{\rm (ii)}] If $M^m$ is biharmonic, then it is a cylinder $M^{m-1}\times\R$, where $M^{m-1}$ is a non-minimal CMC biharmonic hypersurface in $L^m(\varepsilon)$.
\end{itemize}
\end{proposition}
The case (ii) with $\varepsilon=1$ was first obtained in \cite{MR3107694}. We also observe that when $\varepsilon=-1$ a CMC biharmonic hypersurface in $L^m(\varepsilon)$ is actually minimal.

Next, we recall two results of \cite{MR4333889} which are directly related to the analysis that we carry out in the present paper.
\begin{theorem}[\cite{MR4333889}]\label{Th-Maeta-Ou-1} Any constant angle proper biharmonic surface in $L^2(\varepsilon)\times \mathbb{R}$ is an open part of $\s^{1}(1/\sqrt{2})\times\R \subset \s^2\times\R$.
\end{theorem}
\begin{theorem}[\cite{MR4333889}]\label{Th-Maeta-Ou-2}Let $m \geq 2$. Then any totally umbilical biharmonic hypersurface in $L^m(\varepsilon)\times \mathbb{R}$ is totally geodesic.
\end{theorem}
Our first theorem is:
\begin{theorem}\label{Th-constant-angle}
Let $\phi:M^{m}\hookrightarrow L^m(\varepsilon)\times \mathbb{R}$  be a proper biconservative (biharmonic) hypersurface with constant angle $\alpha$. Then, locally, 
$M^{m}$ is a cylinder $M^{m-1}\times \R$ and $\phi(x,t)=(\bar\phi(x),t)$, where $\bar\phi:M^{m-1}\hookrightarrow L^m(\varepsilon)$ is a proper biconservative (biharmonic) hypersurface.
\end{theorem}
The following corollary provides, using different methods, an extension of Theorem~\ref{Th-Maeta-Ou-1}.
\begin{corollary}\label{Cor-1}
Let $m=2$ or $3$. Then any constant angle proper biharmonic hypersurface in $L^m(\varepsilon)\times \mathbb{R}$ is an open part of $\s^{m-1}(1/\sqrt{2})\times\R \subset \s^m\times\R$.
\end{corollary}
To state our next corollary, we recall that, by definition, an {\em isoparametric hypersurface} in a Riemannian manifold is a hypersurface with constant mean curvature whose nearby (locally defined) parallel hypersurfaces also have constant mean curvature. We then have \begin{corollary}\label{Cor-2}
Any  proper biconservative (biharmonic) isoparametric hypersurface in  $L^m(\varepsilon)\times\R$ is a product $M^m=M^{m-1}\times\R$, where $M^{m-1}$ is a proper biconservative (biharmonic) isoparametric hypersurface of $L^m(\varepsilon)$.
\end{corollary}
Next, we turn our attention to the case that \textit{the angle function $\alpha$ is non-constant}. To this end, we begin by outlining some general results obtained by Tojeiro in \cite{MR2738910}. 
More specifically, we recall that hypersurfaces $\phi: M^m \hookrightarrow L^m(\varepsilon) \times \mathbb{R}$ for which $T$ \textit{is everywhere a principal direction} have been studied in \cite{MR2738910} and are fully described by the following construction. We consider the natural embeddings $\s^m\times \R\subset \R^{m+1}\times\R$ and $\h^m\times \R\subset \R^{m+1}_1\times\R$, where $\R^{m+1}_1$ is the Lorentzian space of dimension $m+1$. Now, let {$f_0: M^{m-1} \rightarrow L^m(\varepsilon)$} be an oriented hypersurface 
and $f_s: M^{m-1} \rightarrow L^m(\varepsilon)$ the family of its parallel hypersurfaces, that is,
\begin{equation*}\label{eq-paral-hyp}
f_s(x)=C_\varepsilon(s) f_0(x)+S_\varepsilon(s) N(x),
\end{equation*}
where $N$ is a unit normal vector field to {$f_0$} and
\[
C_\varepsilon(s)=\left\{\begin{array}{l}
\cos s, \text { if } \varepsilon=1 \\
\cosh s, \text { if } \varepsilon=-1
\end{array} \,\,, \quad  \quad S_\varepsilon(s)=\left\{\begin{array}{l}
\sin s, \text { if } \varepsilon=1 \\
\sinh s, \text { if } \varepsilon=-1
\end{array}\right.\right. \,.
\]
We consider $M^m=M^{m-1} \times I$ for some real interval $I$ and immersions
\[
\phi: M^m=M^{m-1} \times I \rightarrow L^m(\varepsilon) \times \mathbb{R}
\]
given by
\begin{equation}\label{eq:phi-tojeiro-family}
\phi(x, s)=f_s(x)+a(s) \partial_t\,,
\end{equation}
where $a: I \rightarrow \mathbb{R}$ is a smooth increasing function.
 We refer to the hypersurfaces of the form \eqref{eq:phi-tojeiro-family} as the {\em Tojeiro family}.  The following result of Tojeiro is important for our investigation:
\begin{theorem}[\cite{MR2738910}]\label{teo:main-tojeiro}
Any map $\phi$ of the Tojeiro family \eqref{eq:phi-tojeiro-family} defines, at regular points, a hypersurface that has $T$ as
a principal direction. Conversely, any hypersurface $\phi: M^m \hookrightarrow L^m(\varepsilon) \times \mathbb{R}$, $m\geq 2$, with $\cos \alpha \neq 0$ everywhere and having T as a principal direction is locally described as in \eqref{eq:phi-tojeiro-family}.
\end{theorem}
Clearly, Proposition~\ref{Prop-CMC} provides strong motivation to study the case where $H$ is non-constant. We observe that if $\cos\alpha=0$ in an open subset, then the hypersurface is locally a cylinder and biconservativity implies that $\grad H$ is a principal direction as an immediate consequence of \eqref{BHP1}. Moreover, $M^m=M^{m-1}\times\R$, where $M^{m-1}$ is a biconservative hypersurface of $L^m(\varepsilon)$. Therefore, when $\cos\alpha\neq 0$, a natural geometric assumption is that $\grad H$ be a principal direction. Our analysis establishes a close relationship between Tojeiro’s hypersurfaces \eqref{eq:phi-tojeiro-family} and biconservative hypersurfaces with $\grad H$ as a principal direction (see also \cite{MR4992706,MR4939182}). More specifically, our first result in this context is:

\begin{theorem}\label{Th-nostro-1}
Let $\phi:M^{m}\hookrightarrow L^m(\varepsilon)\times \mathbb{R}$ be a proper biconservative hypersurface such that $\grad H$ is a principal direction.  
If $\cos\alpha\neq 0$ everywhere, then $M^{m}$ {is locally described as a Tojeiro hypersurface \eqref{eq:phi-tojeiro-family} with $f_0: M^{m-1} \rightarrow L^m(\varepsilon)$ isoparametric.}
\end{theorem}
Proper biconservative hypersurfaces as in Theorem~\ref{Th-nostro-1} do exist and provide {in all dimensions} examples which are  neither slices nor cylinders. Indeed, we have:
\begin{proposition}\label{Th-exist-loc-bic} Let $f_0: M^{m-1} \rightarrow L^m(\varepsilon)$ be any given isoparametric hypersurface. Then the associated Tojeiro family \eqref{eq:phi-tojeiro-family} contains an infinite subfamily of locally defined proper biconservative hypersurfaces whose angle function $\alpha$ is non-constant. 
\end{proposition}
In general, it is difficult to obtain a full qualitative description of the biconservative hypersurfaces constructed in Proposition~\ref{Th-exist-loc-bic}. However, the special case $m=2$ was studied in \cite{MR4939182}, where explicit local parameterizations were determined. 

{It was proved in \cite[Proposition~4.5]{MR2395125} that a totally umbilical, non-CMC hypersurface is biconservative if and only if $m=4$. We describe this type of solutions in more detail in Subsection~\ref{Sub-sec-bic}. However, it is then more interesting to investigate the existence and geometrical properties of \textit{non-totally umbilical} biconservative hypersurfaces. This issue will be treated in Subsection~\ref{Sub-sec-bic-non-umb}, where we investigate the existence of non-totally umbilical solutions and their rigidity. In particular, we shall first prove the following result:}

\begin{theorem}\label{Th-non-umbil-bic}
Let $m\ge2$. There exists a complete, non-cylindrical proper biconservative hypersurface
\[
 \phi:{M}^m
       \hookrightarrow \mathbb H^m\times\mathbb R
       \] 
such that 
\begin{itemize}
\item[{\rm (i)}]  $\phi$ is non-totally umbilical and if $m=4$ has constant scalar curvature ${\rm Scal}=-12$;
\item[{\rm (ii)}] The squared mean curvature function $H^2$ of $\phi$ is non-constant and bounded;
\item[{\rm (iii)}] The Ricci curvature of $\phi$ is bounded from below.
\end{itemize}
\end{theorem}

The biharmonic equation is commonly regarded as a non-flexible system and so it is not surprising that in the literature several rigidity results for biharmonic submanifolds are available (see, for example,  \cite{MR4410183} and \cite[Chapter~7]{MR4265170}).
In the course of the proof of Theorem~\ref{Th-non-umbil-bic} we shall carry out  the explicit construction of the hypersurfaces $\phi$ having the stated properties. These hypersurfaces locally are generated by horospheres through the Tojeiro construction. Moreover, they have two distinct principal curvatures and $\grad H$ as principal direction.  We shall prove that these two properties, under an extra curvature condition, give the following rigidity result. 

\begin{theorem}\label{thm:rigidity}
Let $ \phi:M^4\hookrightarrow \h^4\times\R $ be a proper biconservative hypersurface with constant scalar curvature and at most two different principal curvatures. Assume that $\cos\alpha\neq 0$ everywhere. If $\grad H\neq0$ is a principal direction then, locally, $\phi$ {is congruent to} the hypersurface  of Theorem~\ref{Th-non-umbil-bic}~{\rm (i)}. 
\end{theorem}

\vspace{3mm}

Finally, we use the Omori-Yau maximum principle to prove non-existence of certain  complete proper biharmonic hypersurfaces. More precisely, we have:
\begin{theorem}\label{Th-Max-Principle-application} 
Let $\phi:M^m \hookrightarrow \h^m \times \R$ be a complete, biharmonic hypersurface such that the function $H^2$ is bounded. If ${\rm Ric}(M)$ is bounded from below, then the hypersurface is minimal.
\end{theorem}

\begin{remark}\label{rem:comparison-theorem-1.11}
Theorem~\ref{Th-non-umbil-bic} supplies complete proper biconservative hypersurfaces in $\mathbb H^m\times\mathbb R$
with $H^2$ bounded and Ricci curvature bounded from below.
Thus the global assumptions in Theorem~\ref{Th-Max-Principle-application} do not force minimality in the biconservative setting.
\end{remark}

Throughout the paper, we adopt the sign conventions
\[
R(X,Y)Z=\nabla_X\nabla_YZ-\nabla_Y\nabla_XZ-\nabla_{[X,Y]}Z
\]
for the Riemann curvature tensor and  $\Delta u =-\Div(\grad u)$ for the Laplace operator.

\section{Proofs of the results}
As a preliminary work, we recall here some basic facts concerning the Tojeiro hypersurfaces \eqref{eq:phi-tojeiro-family} (see \cite{MR2738910}).

Let
\begin{equation}\label{eq:rho-def}
\rho(s)=\frac{a'(s)}{\sqrt{1+{a'}^2(s)}}.
\end{equation}
The principal curvatures of $\phi$ are 
\begin{equation}\label{eq-curv-prin}
\big[k^{\phi}_1(s),\ldots,k^{\phi}_{m-1}(s),k^{\phi}_m(s)\big]=\big[-\rho(s) \kappa_1(s),\ldots,-\rho(s) \kappa_{m-1}(s),\rho'(s)\big]
\end{equation} 
where $ \kappa_1(s),\ldots, \kappa_{m-1}(s)$ are the principal curvatures of $f_s$, while $\rho'(s)$ represents the principal curvature with respect to the principal direction $\partial_s$. {Note that, unless $f_0:M^{m-1}\hookrightarrow L^m(\varepsilon)$ is isoparametric, the principal curvatures in \eqref{eq-curv-prin} depend also on $x$. Next, we observe that the mean curvature function $H$ of $\phi$ satisfies 
\begin{equation}\label{eq:aHsH-2}
m \,H=-(m-1)\rho(s) H_s+\rho'(s) \,,
\end{equation}
where $H_s$ is the mean curvature function of $f_s$ and 
\begin{equation*}\label{eq:aHsA2}
|A|^2=\rho^2(s) |A_s|^2+{\rho'}^2(s)\,,
\end{equation*}
where $|A_s|^2=\kappa^2_1(s)+\cdots +\kappa^2_{m-1}(s)$.} Moreover, (see \cite{MR2738910})
{
\begin{equation}\label{eq-rho=sin2alpha}
{\rm (i)}\,T=\frac{a'(s)}{1+{a'(s)}^2}\,\partial_s = \rho \,\cos \alpha \,\partial_s \qquad  {\rm (ii)} \,|T|^2=\rho^2(s)=\sin^2 \alpha\,.
\end{equation}
If we assume that the initial hypersurface $f_0$ is isoparametric, then $H_s$ depends only on $s$ and, taking into account \eqref{eq-rho=sin2alpha}(ii),} a routine computation shows that the biharmonicity condition stated in Proposition~\ref{Prop-Ou} is equivalent to the following system:
\begin{equation}\label{eq-biharm-system-2}
\begin{cases}
H+\frac{m-1}{m} \rho H_s-\frac{1}{m} \rho^{\prime} =0 \\
H^{\prime} \rho^{\prime}+\frac{m}{2} H H^{\prime}+(m-1) \varepsilon H \rho =0\,\,\quad {\rm Tangent}\\
-\Delta H-  H (\rho^2 |A_s|^2+{\rho'}^2) + \varepsilon  (m-1) H \rho^2=0\,\,\quad {\rm Normal}\,.
\end{cases}
\end{equation}
Now, we begin the proofs of the results stated in the introduction.
\begin{proof}[Proof of Theorem~\ref{Th-constant-angle}]
First, let $\cos\alpha \equiv 0$. Then, locally,  $M^m$ is a cylinder $M^{m-1}\times\R$, where $M^{m-1}$ is a hypersurface of $L^m(\varepsilon)$. Moreover, the principal curvatures of $M^m$ are the $m-1$ principal curvatures $\kappa_1,\ldots\kappa_{m-1}$ of $M^{m-1}$ in $L^m(\varepsilon)$ and $\kappa_m=0$. Then 
\[
\begin{split}
H=&\frac{m-1}{m} H_{M^{m-1}}\,,\quad |A|^2=|A_{M^{m-1}}|^2\,, \\
\grad H=&\frac{m-1}{m}\grad_{M^{m-1}}(H_{M^{m-1}})\,.
\end{split}
\] 
Inspection of \eqref{BHP1} and \eqref{BHP2} shows that $M^m$ is biconservative (biharmonic) in $L^m(\varepsilon)\times\R$ if and only if $M^{m-1}$ is biconservative (biharmonic) in $L^m(\varepsilon)$.

Therefore, it only remains to prove that no proper biconservative hypersurface can have
$\cos\alpha$ equal to a \textit{nonzero} constant. 
First, we recall that, 
using the fact that $\partial_t$ is parallel on $L^m(\varepsilon)\times \mathbb R$, a direct computation gives
\begin{eqnarray}
\nabla_XT&=&\cos\alpha \,A(X),\label{eq:nablaXT}\\
X(\cos\alpha)&=&-\langle A(X),T\rangle=- \langle X,A (T)\rangle \label{eq:XTheta}
\end{eqnarray}
for every tangent vector field $X$ on $M^m$. It follows from \eqref{eq:XTheta} that $A(T)=0$.
If $T=0$, then $|\cos\alpha|=1$ and $M^m$ is an open part of a
totally geodesic slice, hence it is not proper. Thus it remains to
exclude the case $0<|\cos\alpha|<1$.
Since $T$ is a principal direction and $\cos\alpha\neq 0$, by
Theorem~\ref{teo:main-tojeiro} the hypersurface is locally a Tojeiro hypersurface
\[
\phi(x,s)=\bigl(f_s(x),a(s)\bigr),
\]
where $f_s$ is the family of parallel hypersurfaces of an oriented
hypersurface $f_0:M^{m-1}\longrightarrow L^m(\varepsilon)$.
Since $\alpha$ is constant and we are studying the case $0<|\cos\alpha|<1$, \eqref{eq-rho=sin2alpha} implies that $\rho$ is a nonzero
constant. Thus \eqref{eq:aHsH-2} gives
\begin{equation}
mH(x,s)=-(m-1)\rho H_s(x,s),
\label{eq:constant-angle-H}
\end{equation}
where $H_s$ denotes the mean curvature of the parallel hypersurface
$f_s$. Notice that, at this stage, $H_s$ is allowed to depend on both
$x$ and $s$.

The biconservativity equation is
\[
2A(\grad H)+mH\grad H
+2(m-1)\varepsilon H\cos\alpha\,T=0.
\]
Taking the scalar product with $T$ and using $A(T)=0$, we obtain
\begin{equation}\label{component-T-biconservative}
mH\,T(H)
+2(m-1)\varepsilon H\cos\alpha\,|T|^2=0.
\end{equation}
Since the hypersurface is non-minimal, there is a nonempty open set on which
$H\neq 0$. On this set,
\[
m\,T(H)+2(m-1)\varepsilon\cos\alpha\,|T|^2=0.
\]
{Now, using \eqref{eq-rho=sin2alpha}, we easily find that condition \eqref{component-T-biconservative} becomes}
\begin{equation}\label{deriv-H-wrt-s1}
\partial_s H
=-\frac{2(m-1)}{m}\varepsilon\rho.
\end{equation}
Differentiating \eqref{eq:constant-angle-H} with respect to $s$ and
replacing in \eqref{deriv-H-wrt-s1}, we deduce
\begin{equation}
{\partial_s H_s=2\varepsilon.}
\label{eq:Hs-derivative}
\end{equation}

We now derive a contradiction using the geometry of parallel
hypersurfaces in a space form. {Fix $x\in M^{m-1}$ and let
$\kappa_1^0,\ldots,\kappa_{m-1}^0$ be the principal curvatures of
$f_0$ at $x$. On any interval free of focal points, the corresponding
principal curvatures of $f_s$ are
\[
\kappa_i(s)=
\frac{
\kappa_i^0 C_\varepsilon(s)+\varepsilon S_\varepsilon(s)
}{
C_\varepsilon(s)-\kappa_i^0 S_\varepsilon(s)
},
\qquad i=1,\ldots,m-1.
\]
}
A direct computation gives
\begin{equation}\label{fund-dif-eqn-ki}
\kappa_i'=
\varepsilon+\kappa_i^2.
\end{equation}
Since
\[
H_s(x,s)=\frac{1}{m-1}
\sum_{i=1}^{m-1}\kappa_i(s),
\]
it follows that
\[
\partial_s H_s
=
\varepsilon+
\frac{1}{m-1}\sum_{i=1}^{m-1}\kappa_i^2.
\]
Combining this identity with \eqref{eq:Hs-derivative}, we obtain
\begin{equation}
\frac{1}{m-1}\sum_{i=1}^{m-1}\kappa_i^2=\varepsilon.
\label{eq:kappa-square-identity}
\end{equation}

If $\varepsilon=-1$, equation~\eqref{eq:kappa-square-identity} is clearly
impossible. If $\varepsilon=1$, then
\[
Q(s):=\frac{1}{m-1}\sum_{i=1}^{m-1}\kappa_i(s)^2\equiv 1\,.
\]
On the other hand, using \eqref{fund-dif-eqn-ki},
\[
Q''
=
\frac{2}{m-1}
\sum_{i=1}^{m-1}
(1+\kappa_i^2)(1+3\kappa_i^2)>0
\]
which contradicts $Q\equiv 1$.

{Then the case $0<|\cos\alpha|<1$ cannot occur and so the proof is completed.}
\end{proof}
\begin{proof}[Proof of Proposition~\ref{Prop-CMC}] (i) From \eqref{BHP1} we see that a non-minimal CMC  hypersurface $M^m$ in $L^m(\varepsilon)\times\R$ is  biconservative if and only if
\[
\cos\alpha\, T=0\,.
\]
As in the proof of Theorem~\ref{Th-constant-angle}, if $\cos\alpha=0$ the hypersurface is a cylinder $M^{m-1}\times\R$, where $M^{m-1}$ is a non-minimal CMC hypersurface of $L^m(\varepsilon)$. The case $T=0$ is not admissible because otherwise the hypersurface would be a totally geodesic slice. The proof of statement (ii) is analogous.
\end{proof}
\begin{proof}[Proof of Corollary~\ref{Cor-1}]
This corollary is an immediate consequence of Theorem~\ref{Th-constant-angle} and the classification of proper biharmonic hypersurfaces in $L^m(\varepsilon)$, $\varepsilon \pm 1,\,m=2,3$. 
\end{proof}
\begin{proof}[Proof of Corollary~\ref{Cor-2}]
It was proved in \cite{arXiv:2411.11506} that any isoparametric hypersurface in $L^m(\varepsilon)\times\R$ has constant principal curvatures and constant angle $\alpha$. Then the conclusion follows immediately from Theorem~\ref{Th-constant-angle}.
\end{proof}

\begin{proof}[Proof of Theorem~\ref{Th-nostro-1}]
By assumption, $\cos\alpha\neq 0$ and $\phi:M^{m}\hookrightarrow L^m(\varepsilon)\times \mathbb{R}$ is a proper biconservative hypersurface such that $\grad H$ is a principal direction. The case that
\[
-2 m A( \grad H)-m^2 H \grad H=0
\]
in an open set is not acceptable, otherwise \eqref{BHP1} implies $T=0$ and so the hypersurface would be a totally geodesic slice.

Then we deduce from \eqref{BHP1} that $T$ is proportional to $\grad H$ and so is a principal direction. Therefore, {locally} the hypersurface is in the Tojeiro family \eqref{eq:phi-tojeiro-family}. Thus $\grad H$ is proportional to $\partial_s$ and we deduce that $H$ depends only on $s$. Therefore, \eqref{eq:aHsH-2} implies that $H_s$ depends only on $s$ meaning that all the parallel hypersurfaces $f_s$ have constant mean curvature. Consequently,  $f=f_0$ is isoparametric and the proof is completed. 
\end{proof}
\begin{proof}[Proof of Proposition~\ref{Th-exist-loc-bic}] From the first equation of \eqref{eq-biharm-system-2} we have
\[
H=-\,\frac{m-1}{m} \rho H_s+\frac{1}{m} \rho^{\prime}\,.
\]
Inserting this into the second equation of \eqref{eq-biharm-system-2} we deduce that the condition of biconservativity is equivalent to the following second order ODE for the function $\rho(s)$:
\begin{small}
\begin{equation}\label{eq-bic-rho-ODE}
\rho''=\frac{(m-1) \left(\rho  \rho ' \left(-3 H_s'+(m-1)H_s^2+2 \epsilon \right)-\left((m-1)H_s \rho ^2 \left(2 \varepsilon -H_s'\right)\right)-3 H_s \rho '^2\right)}{(m-1) H_s \rho-3 \rho '}
\end{equation}
\end{small}
Now, we can choose initial conditions
\begin{equation*}\label{eq-init-conditions}
\rho(0)= \rho_0 \quad (0< \rho_0 < 1)\,, \quad \rho'(0)= \rho_1\neq 0
\end{equation*}
in such a way that 
\begin{equation}\label{eq-cond-ex-uniqueness-th-ODE}
(m-1) H_s(0) \rho_0-3 \rho_1 \neq 0\quad{\rm and}\quad (m-1) H_s(0) \rho_0- \rho_1 \neq 0\,.
\end{equation}
With any such choice, the standard existence and uniqueness theorem for the Cauchy problem guarantees that the ODE \eqref{eq-bic-rho-ODE} admits a local solution defined on an open neighbourhood of $s=0$. Of course, there are infinitely many possible choices for which \eqref{eq-cond-ex-uniqueness-th-ODE} is satisfied. Moreover, since $\rho(s)$ is non-constant, the angle function $\alpha$ is also non-constant and so the proof is ended. 

For the sake of completeness, we observe that, using \eqref{eq:rho-def}, we can rewrite the ODE \eqref{eq-bic-rho-ODE} as a third order ODE for the function $a(s)$. In this case, conditions \eqref{eq-cond-ex-uniqueness-th-ODE} become
\begin{align*}
(m-1)H_s(0)\left (a'^3(0) +a'(0) \right )-3 a''(0)\neq 0\quad {\rm and}\\
(m-1)H_s(0)\left (a'^3(0) +a'(0) \right )-a''(0)\neq 0 \,.
\end{align*}
\end{proof}
We point out that given initial conditions satisfying \eqref{eq-cond-ex-uniqueness-th-ODE}, one can plot a numerical approximation of the corresponding biconservative solution of \eqref{eq-bic-rho-ODE} using a suitable software, although a complete qualitative analysis of these solutions appears to be difficult. Next, we explore in greater detail two important cases.
\subsection{Totally umbilical solutions}\label{Sub-sec-bic} 
{As we pointed out in the introduction, a totally umbilical, non-CMC hypersurface is biconservative if and only if $m=4$. A thorough study of totally umbilical hypersurfaces in $L(\varepsilon) \times \R$ has been carried out in \cite{MR4493653}. Some basic examples can be fitted into the Tojeiro construction as follows and for the sake of completeness we report here the details which are more relevant for our discussion.}
 
{First, let $\varepsilon=1$ and $f_0: M^{m-1} \rightarrow \mathbb{S}^m$ be the equator. In this case we have
\[
 H_s=\tan s \quad {\rm and} \quad |A_s|^2=(m-1) \tan^2 s \,.
\] 
Therefore, system \eqref{eq-biharm-system-2} becomes 
 \begin{equation}\label{eq-system-con-soluzione}
\begin{cases}
H+\frac{m-1}{m} \rho\tan s-\frac{1}{m} \rho^{\prime} =0 \\
H^{\prime} \rho^{\prime}+\frac{m}{2} H H^{\prime}+(m-1)  H \rho =0\quad {\rm Tangent}\\
-\Delta H-  H \big[(m-1) \rho^2 (-1+\tan^2 s)+{\rho'}^2 \big]=0\quad {\rm Normal}\,.
\end{cases}
\end{equation}
Using \eqref{eq-curv-prin} we find that the principal curvatures are
\[
k_1^\phi=\dots=k_{m-1}^\phi=-\,\rho\, \tan s\, \qquad k_m^\phi=\rho' 
\]
and so the hypersurface is totally umbilical if and only if
\[
-\,\rho\, \tan s=\rho'
\]
which implies
\begin{equation*}\label{eq-rho-explicit}
\rho(s)= c \cos s \,,
\end{equation*}
where $0<c\leq 1$ is a constant. Direct substitution in \eqref{eq-system-con-soluzione} gives
\begin{equation*}\label{eq-system-con-rho-data}
\begin{cases}
H(s)=-c \sin s \\
-\frac{1}{4}\, c^2 (m-4) \sin (2 s)=0\quad {\rm Tangent}\\
-\frac{1}{2} \,c\, m \sin s \left(3 c^2 \cos (2 s)+c^2-2\right)=0\quad {\rm Normal}
\end{cases}
\end{equation*}
which confirms biconservativity if and only if $m=4$.}
{Therefore, we have proper biconservative hypersurfaces 
\[
\phi: M^4=\s^{3} \times I \subset \R^4 \times I \rightarrow \s^4 \times \mathbb{R} \subset \R^5 \times I
\]
given by
$$
\phi(x, s)=\big ( \cos s \, x, \sin s, a(s) \big )
$$
and the explicit expression of the function $a(s)$ can be obtained by integration of the first order ODE
\begin{equation*}\label{eq-ODE-a-rho}
c\,\cos s = \frac{a'(s)}{\sqrt{1+a'^2(s)}}\,.
\end{equation*}
We obtain, up to a vertical translation, the following solutions: 
\begin{itemize}
\item[(i)] Case $0<c<1$: 
\[
a(s)=\sinh ^{-1}\left(\frac{c \, \sin s}{\sqrt{1-c^2}}\right)\,, \quad s \in \left (- \frac{\pi}{2}, \frac{\pi}{2} \right )\,.
\]
\item[(ii)] Case $c=1$:
\[
a(s)=\log (\sin s)\,, \quad 0<s < \frac{\pi}{2}\,.
\]
\end{itemize}
}
{Since the biconservativity condition is a smooth tensorial equation any solution of the type (i) extends smoothly across the two poles and produces a \textit{compact} proper biconservative hypersurface which actually is a graph $\phi: \s^4 \hookrightarrow \s^4 \times \R$.}
{More explicitly, let $x=\left ( x_1,\ldots, x_5\right )\in \s^4 \subset \R^5$. Then we have:}
\[
\phi(x)=\left (x,\sinh ^{-1}\left(\frac{c\, x_5}{\sqrt{1-c^2}}\right) \right)\,.
\]
{Similarly to the case (i), the solution (ii) extends smoothly across the North pole and produces a totally umbilical proper biconservative hypersurface which actually is the graph $\phi: \s^4_+ \hookrightarrow \s^4_+ \times \R \subset \s^4 \times \R$ defined by
\[
\phi(x)=\left (x,\log x_5 \right)\,.
\]
We show that this graph is complete. To this purpose, we first observe that the induced metric $\tilde{g}$ on $\s^4_+$ can be described as follows:
\[
\s^4_+=\left ( \s^3 \times (0,\pi/2], \tilde{g}\right )\,,
\]
where
\[
\tilde{g}=\cos^2 s \, g_{\s^3}+ \frac{1}{\sin^2 s}\,ds^2\,.
\]
Now, all the meridians emanating from the North pole are geodesics of infinite length. Therefore, the completeness of $\tilde{g}$ follows readily from the Hopf-Rinow theorem.
} 
\vspace{3mm}

\subsection{Non-totally umbilical solutions: the proofs of Theorems~\ref{Th-non-umbil-bic} and \ref{thm:rigidity}}\label{Sub-sec-bic-non-umb} 
{
\begin{proof}[Proof of Theorem~\ref{Th-non-umbil-bic}] We construct explicitly the required hypersurfaces. We use the upper half-space model
\[
\h^m=\{(x,y):x\in\R^{m-1},\ y>0\}\,,
\qquad
 g_{\h^m}=\frac{dx^2+dy^2}{y^2}\,.
\]
We take the parallel horospheres
\[
f_s(x)=(x,e^s)
\]
which satisfy
\[
\kappa_1(s)=\cdots=\kappa_{m-1}(s)=1\,,
\qquad H_s=1
\]
and consider a Tojeiro hypersurface
\[
\phi(x,s)=\bigl(f_s(x),a(s)\bigr)\,.
\]
According to \eqref{eq-curv-prin}, the principal curvatures of $\phi$ are
\begin{equation}\label{eq:k1kminrho}
k_1^\phi=\cdots=k_{m-1}^\phi=-\rho\,,
\qquad k_m^\phi=\rho'
\end{equation}
and hence its mean curvature is
\begin{equation*}
 H=\frac{1}{m}\big ( \rho'-(m-1)\rho \big )\,.
 \label{eq:mean}
\end{equation*}
In this setting, the biconservative equation reduces to
\begin{equation}
 H'\rho'+\frac m2 HH'-(m-1)H\rho=0\,.
 \label{eq:BC}
\end{equation}
We observe that, according to \eqref{eq:k1kminrho},  if $ \rho(s)=e^{- s}$ the hypersurface is totally umbilical. Therefore, it is natural to seek non-totally umbilical solutions of the type
\begin{equation*}
 \rho(s)=e^{\beta s},\qquad \beta\in \R,\;  \beta\neq -1\,,
 \label{eq:rhoexp}
\end{equation*}
which are defined for all $s$ such that $e^{\beta s}<1$.
Then
\begin{equation}\label{eq:r1HH1}
\rho'=\beta\rho\,,
\qquad
H=\frac{\beta-(m-1)}{m}\rho\,,
\qquad
H'=\beta H\,.
\end{equation}
If $\beta=m-1$ the hypersurface is minimal. Since
\[
a'(s)=\frac{\rho(s)}{\sqrt{1-\rho(s)^2}}\,,
\]
we obtain, up to a vertical translation,
\begin{equation*}
 a(s)=\frac1\beta\arcsin\bigl(e^{\beta s}\bigr)\,.
 \label{eq:a}
\end{equation*}
Substituting \eqref{eq:r1HH1} into \eqref{eq:BC} tells us that the biconservativity condition becomes
\begin{equation*}
 3\beta^2-(m-1)\beta-2(m-1)=0
 \label{eq:betaquad}
\end{equation*}
whose solutions are the $\beta_{m\pm}$ defined by
\begin{equation*} \label{eq:betapm}
 \beta_{m\pm}=
 \frac{(m-1)\pm\sqrt{(m-1)(m+23)}}{6}\,.
\end{equation*}
We define
\begin{equation*}\label{eq:beta-choice-intro}
\beta_m=
\begin{cases}
\beta_{2-}=-\dfrac23,&m=2,\\[2mm]
\beta_{m+},&m\ge3.
\end{cases}
\end{equation*}
Now, we set
\begin{equation*}
\phi_m(x,s)=
 \left(
 x,e^s,
 \frac{1}{\beta_m}\arcsin(e^{\beta_m s})
 \right)\,.
 \label{eq:localimm}
\end{equation*}
We conclude that these hypersurfaces are proper biconservative. Moreover, by construction, it is clear that
$\phi_{m}$ has two distinct principal curvatures $k_i^{\phi} $, $i=1,2$, with multiplicities  $(m-1)$ and $1$ respectively. Moreover, according to \eqref{eq:k1kminrho}
\begin{equation*}\label{eq:prop-tra-curv-prin}
k_2^{ \phi}=-\beta_{m}\,k_1^{ \phi}\,.
\end{equation*}

Now, let us focus on completeness.  The induced metric is
\begin{equation*}
 g_{M}=e^{-2s}g_{\R^{m-1}}
 +\frac{ds^2}{1-e^{2\beta_m s}}\,.
 \label{eq:localmetric}
\end{equation*}
The endpoint where $\rho\to1$ lies at finite intrinsic distance. Hence the
$s$-graph itself is not complete. This endpoint is, however, only a vertical
tangent point and so we can extend the hypersurface by reflection. Indeed, let
\[
\lambda=|\beta_{m}|\,.
\]
After possibly reversing the vertical $\R$-coordinate, we introduce a new
parameter $t$ by
\begin{equation*}
 \rho(s)=\sin(\lambda t)\,,
 \qquad
 0<t<\frac{\pi}{\lambda}\,.
 \label{eq:rhot}
\end{equation*}
Since $\rho(s)=e^{\beta_m s}$, we have
\begin{equation}
 s(t)=\frac{1}{\beta_m}
 \log\left({\sin(\lambda t)}\right)\,.
 \label{eq:st}
\end{equation}
Therefore the local graph extends to
\begin{equation*}
 \phi_m(x,t)=
 \left(
 x,
 \sin(\lambda t)^{1/\beta_m}\,,
 t
 \right)\,,
 \qquad
 x\in\R^{m-1},\quad 0<t<\frac{\pi}{\lambda}\,.
 \label{eq:globalimm}
\end{equation*}
The last component is $t$, so the immersion is injective. Moreover, we observe that
\begin{equation*}
 H(t)=\frac{\beta_m-(m-1)}{m}\sin(\lambda t)
 \label{eq:Hglobal}
\end{equation*}
and so property (ii) of the statement of Theorem~\ref{Th-non-umbil-bic} holds. Next, differentiating \eqref{eq:st}, we have
\[
ds=\frac{\lambda}{\beta_m}\cot(\lambda t)dt\,.
\]
Thus
\begin{equation}
 g_{ M}
 =\sin(\lambda t)^{-2/\beta_m}
 g_{\R^{m-1}}
 +\frac{dt^2}{\sin^2(\lambda t)}\,.
 \label{eq:tmetric}
\end{equation}
Next, we introduce the global coordinate
\begin{equation*}
 r=\frac1\lambda\log\tan\frac{\lambda t}{2}\,.
 \label{eq:rcoord}
\end{equation*}
Then $r\in\R$ and
\[
dr=\frac{dt}{\sin(\lambda t)},
\qquad
\sin(\lambda t)=\sech(\lambda r)\,.
\]
Consequently \eqref{eq:tmetric} becomes
\begin{equation}
 g_{ M}
 =f^2(r)\,g_{\R^{m-1}}+dr^2
 \label{eq:rmetric}
\end{equation}
where
\[
f(r)=\cosh(\lambda r)^{1/\beta_m}\,.
\]
Now, since $f(r)$ is a smooth, positive function on the whole $\R$, we easily conclude that these hypersurfaces obtained by reflection of Tojeiro's graphs are complete. 

In summary, to end the proof of Theorem~\ref{Th-non-umbil-bic} it only remains to show that the Ricci curvature of ${M}^m$ is bounded from below and $\Scal=-12$ for $\phi_4$. Indeed, let
\[
E_i= \frac{1}{f}\,\frac{\partial}{\partial x_i}\,, \, i=1,\dots,m-1 \quad {\rm and}\quad E_m=\frac{\partial}{\partial r}\,.
\]
Then $\left \{ E_1, \dots,E_m\right \}$ is a global orthonormal frame field on ${M}^m$ with respect to the induced metric $ g_{ M}$. 
Since $ g_{ M}$ is a warped product metric with flat fiber of dimension $m-1$, one has (see, for example, \cite[p. 209]{MR719023})
\[
\Ric(\partial_r,\partial_r)=-(m-1)\frac{f''}{f}\,,
\]
and, if $1 \le i \le m-1$,
\[
\Ric(E_i,E_i)
=-\frac{f''}{f}
 -(m-2)\left(\frac{f'}f\right)^2\,.
\]
Moreover, all mixed components vanish in this adapted orthonormal frame.
Therefore, using the explicit expression of $f$ in \eqref{eq:rmetric}, we obtain:
\begin{equation*}
 \Ric(\partial_r,\partial_r)
 =-(m-1)\left[1+(\beta_m-1)\sech^2(\lambda r)\right]
 \label{eq:Ricrad}
\end{equation*}
and, for $i=1,\dots,m-1$,
\begin{equation*}
 \Ric(E_i,E_i)
 =-(m-1)+(m-1-\beta_m)\sech^2(\lambda r)\,.
 \label{eq:Rictan}
\end{equation*}
It follows readily that the Ricci curvature of ${M}^m$ is bounded from below. Finally, taking the trace of the Ricci tensor we obtain
\begin{equation*}
 \Scal
 =-m(m-1)
 +(m-1)(m-2\beta_m)\sech^2(\lambda r)\,.
 \label{eq:scalar}
\end{equation*}
Therefore the scalar curvature is constant precisely when
\[
m-2\beta_m=0\,.
\]
More specifically, this occurs only when $m=4$ and $\beta_{4}=2$. In this case we have $\Scal=-12$ and so the proof of Theorem~\ref{Th-non-umbil-bic} is complete.

 \end{proof}
}

\begin{proof}[Proof of Theorem~\ref{thm:rigidity}] 
First, we recall that, with respect to a local orthonormal tangent frame field $\{E_1,\ldots,E_4\}$ on $M$ the Gauss equation gives, for $i\neq j$,
\[
\langle R(E_i,E_j)E_j,E_i\rangle
=
\langle \overline R(E_i,E_j)E_j,E_i\rangle
+
\langle AE_i,E_i\rangle\langle AE_j,E_j\rangle
-
\langle AE_i,E_j\rangle^2.
\]
Summing over all $1\leq i<j\leq 4$, we obtain
\[
\Scal_M
=
2\sum_{i<j}
\langle \overline R(E_i,E_j)E_j,E_i\rangle
+(\operatorname{tr}A)^2-|A|^2.
\]
Using the expression of the curvature operator in the product $\Hh^4\times\R$:
\begin{align*}\label{eq:tildeR}
\bar R(\bar X,\bar Y)\bar Z=-\{\langle \bar Y,\bar Z\rangle \bar X-\langle \bar X,\bar Z\rangle \bar Y -\langle
\bar Y,\partial_t\rangle\langle \bar Z,\partial_t\rangle \bar X
+\langle \bar X,\partial_t\rangle\langle \bar Z,\partial_t\rangle \bar Y\nonumber\\
+\langle \bar X,\bar Z\rangle\langle \bar Y,\partial_t\rangle \partial_t -\langle
\bar Y,\bar Z\rangle\langle \bar X,\partial_t\rangle \partial_t\},
\end{align*}
where $\bar X, \bar Y, \bar Z$ are vector fields on ${\h}^4\times\mathbb{R}$,
we have
\[
2\sum_{i<j}
\langle \overline R(E_i,E_j)E_j,E_i\rangle= -6\left (1+\cos^2 \alpha \right ),
\]
hence 
\begin{equation}\label{eq:scal-general}
 \Scal_M
 =
 -6(1+\cos^2\alpha)+(\operatorname{tr}A)^2-|A|^2.
\end{equation}
Next, we proceed to the proof.
For simplicity, we  call $u=\cos\alpha$ and $v=|T|=\sin\alpha$. 
We recall that for hypersurfaces in $\Hh^4\times\R$ the biconservative condition \eqref{BHP1} becomes
\begin{equation}\label{BC}
 2A(\grad H)+4H\grad H-6HuT=0.
\end{equation}
{We first exclude the totally umbilical case. Suppose $A=H\,\mathrm{Id}$. Since $H\ne0$, equation \eqref{BC} gives
\[
\grad H=uT.
\]
On the other hand, \eqref{eq:XTheta} gives
\[
\grad u=-HT.
\]
Therefore
\[
\grad(H^2+u^2)=0.
\]
For a totally umbilical hypersurface, \eqref{eq:scal-general} becomes
\[
\Scal_M=-6(1+u^2)+12H^2.
\]
Since both $\Scal_M$ and $H^2+u^2$ are constant, $H^2$ is constant. Because $H\ne0$, this forces $H$ to be constant on each connected neighborhood, contradicting $\grad H\ne0$. Hence the hypersurface has exactly two distinct principal curvatures.}
Since $\grad H$ is principal, there is a function $\kappa$ such that
\[
 A(\grad H)=\kappa\,\grad H\,.
\]
Then \eqref{BC} becomes
\[
 (2\kappa+4H)\grad H=6HuT\,.
\]
By hypothesis $H\neq0$, and $u\neq0$. Therefore, either $T$ is identically zero and the hypersurface is a totally geodesic slice or, in a neighborhood $U$ of $M$, $T\neq 0$ and it is a principal direction. Since we are assuming that the hypersurface is non-minimal, we are left with $T\neq 0$ being a principal direction on $U$ open subset of $M$.
We put
\[
 E_4=\frac{T}{v},\qquad v=|T|>0\,.
\]
Then
\[
 AE_4=\mu E_4
\]
where $\mu$ is the principal curvature in the direction of $E_4$.
Since there are exactly two distinct principal curvatures, the
multiplicity $r$ of $\mu$ is one of
\[
 r=1,2,3.
\]
Let $\lambda\neq\mu$ denote the other principal curvature and let $ q=4-r$
be its multiplicity. Thus
\[
 \operatorname{spec}(A)
 =
 \{\underbrace{\mu,\ldots,\mu}_{r},
   \underbrace{\lambda,\ldots,\lambda}_{q}\}.
\]
Now, we set
\begin{equation}\label{xydef}
 x=\frac{\lambda}{v},
 \qquad
 y=\frac{\mu}{v}.
\end{equation}
Taking $X=E_4$ in \eqref{eq:XTheta} gives
\[
 E_4(u)=-\mu v
\]
and from $u^2+v^2=1$ we obtain
\begin{equation}\label{vprime}
 E_4(v)=u\mu=uyv.
\end{equation}
Next, we choose a unit
$\lambda$--principal vector $X$ orthogonal to $E_4$. Since
$X\perp T$, then \eqref{eq:nablaXT} gives
\[
 v\nabla_XE_4=u\lambda X.
\]
Putting $Y=E_4$ in the Codazzi equation
\[
 (\nabla_XA)Y-(\nabla_YA)X
 =
 -u\bigl(\langle Y,T\rangle X-\langle X,T\rangle Y\bigr)
\]
 and taking the scalar product with $X$ yields
\begin{equation}\label{eq:E4lambda}
 E_4(\lambda)
 =
 \frac{u\lambda}{v}(\mu-\lambda)+uv.
\end{equation}
Using $\lambda=xv$, $\mu=yv$ and \eqref{vprime} in \eqref{eq:E4lambda}, we obtain
\begin{equation}\label{xprime}
{E_4(x)=u(1-x^2).}
\end{equation}
For later use, introduce
\begin{equation}\label{zdef}
 z=\frac{E_4(y)}{u}.
\end{equation}
Since
\[
 \operatorname{tr}A=r\mu+q\lambda,
 \qquad
 |A|^2=r\mu^2+q\lambda^2,
\]
and $u^2=1-v^2$, using \eqref{eq:scal-general} we obtain
\begin{equation*}\label{scalar}
 \Scal_M=-12+v^2{\mathcal F}_r(x,y),
\end{equation*}
where
\begin{equation}\label{Fr}
 {\mathcal F}_r(x,y)
 =
 6+r(r-1)y^2+q(q-1)x^2+2rqxy.
\end{equation}
Assume now that $\Scal_M=S_0$ is constant. Differentiating
\[
 v^2{\mathcal F}_r(x,y)=S_0+12
\]
along $E_4$ and using
\eqref{vprime}, \eqref{xprime}, and \eqref{zdef}, we obtain, after dividing by the nonzero factor $uv^2$,
\begin{equation}\label{Ar}
 {\mathcal A}_r(x,y,z):=
 2y{\mathcal F}_r+(1-x^2)({\mathcal F}_r)_x+z({\mathcal F}_r)_y=0.
\end{equation}
The mean curvature can be written as
\[
 H=vh,
 \qquad
 h=h_r(x,y):=\frac{ry+qx}{4}.
\]
Using again \eqref{vprime}, \eqref{xprime}, and \eqref{zdef},
\begin{equation}\label{E4H}
 E_4(H)
 =
 uv\left(
 yh+\frac q4(1-x^2)+\frac r4z
 \right).
\end{equation}
Because $\grad H$ is parallel to $E_4$, the biconservative condition \eqref{BC} reduces to
\[
 (\mu+2H)E_4(H)-3Huv=0
\]
that, together with \eqref{E4H},  up to a nonzero factor $uv^2$, gives
\begin{equation}\label{Br}
 {\mathcal B}_r(x,y,z):=
 (y+2h)
 \left(
 yh+\frac q4(1-x^2)+\frac r4z
 \right)-3h=0.
\end{equation}
Thus, for each $r=1,2,3$, the functions $x,y,z$ satisfy the two
algebraic relations
\[
 {\mathcal  A}_r(x,y,z)=0,
 \qquad
 {\mathcal B}_r(x,y,z)=0.
\]
Eliminating $z$ from \eqref{Ar} and \eqref{Br} gives a polynomial
relation
\[
 {\mathcal P}_r(x,y)=0.
\]
Up to multiplication by a nonzero constant, these polynomials are
\begin{align}
 {\mathcal P}_1(x,y)
 &=
 x^4-x^3y-x^2y^2+5x^2+xy^3+4xy+3y^2,
\notag\\
 {\mathcal P}_2(x,y)
 &=
 x^4-3x^2y^2+5x^2+2xy^3+11xy+11y^2,
\notag\\
 {\mathcal P}_3(x,y)
 &=
 x^4+3x^3y-9x^2y^2+5x^2+5xy^3+32xy+71y^2.
\notag
\end{align}
Since ${\mathcal P}_r(x,y)$ is identically zero on $U$, differentiating along $E_4$
and dividing by $u\neq0$ gives
\begin{equation*}\label{Dr}
 {\mathcal D}_r(x,y,z):=
 (1-x^2)({\mathcal P}_r)_x+z({\mathcal P}_r)_y=0.
\end{equation*}
Eliminating $z$ from ${\mathcal A}_r(x,y,z)=0$ and ${\mathcal D}_r(x,y,z)=0$ gives a polynomial condition ${\mathcal G}_r(x,y)=0$. Then eliminating $y$
 from ${\mathcal P}_r(x,y)=0$ and ${\mathcal G}_r(x,y)=0$ produces a polynomial condition in $x$:
\[
 {\mathcal R}_r(x)=0.
\]
A direct expansion gives
\begin{align}
 {\mathcal R}_1(x)
 &=
 x^6(x^2-1)
 \bigl(150x^4-156x^2+55\bigr),
\notag\\
 {\mathcal R}_2(x)
 &=
 x^5(x^2-1)
 \bigl(
 26244x^{12}+15876x^{10}-257067x^8+94842x^6
 \notag\\[-1mm]
 &\hspace{35mm}
 +731421x^4-957352x^2+346060
 \bigr),
\notag\\
 {\mathcal R}_3(x)
 &=
 x^5(x^2-2)
 \Bigl(
 762048x^{12}-997488x^{10}-13970736x^8+28751214x^6
 \notag\\[-1mm]
 &\hspace{24mm}
 +55828234x^4
 -189031820x^2+131696125
 \Bigr).\notag
\end{align}
Now, the only fact needed is that each ${\mathcal R}_r$
is a nonzero polynomial. Therefore if $U$ is connected then $x$ is constant
on $U$. Equation \eqref{xprime}, together with $u\neq0$, then forces
\begin{equation*}\label{xpm1}
{x^2=1.}
\end{equation*}
This is the main rigidity step.
We now exclude  the multiplicities $r=2$ and $r=3$. Since $x^2=1$, a direct check shows that $x=\pm 1$ is not a solution of ${\mathcal R}_3(x)=0$, thus the value $r=3$ is not allowed. If $r=2$, then $x=\pm 1$ is a solution of ${\mathcal R}_2(x)=0$. 
But since $x$ is constant, the polynomial relation ${\mathcal P}_r(x,y)=0$ shows
that $y$ takes values in a finite set; hence $y$ is also constant on
the connected set $U$. Therefore from \eqref{zdef}
\[
 z=0.
\]
We now substitute $x=\pm1$ and $z=0$ into the original polynomials
${\mathcal A}_r(x,y,z)$ and ${\mathcal B}_r(x,y,z)$.
For $r=2$ we obtain, when $x=1$,
\[
 {\mathcal A}_2=4y(y+2)^2,
 \qquad
 {\mathcal B}_2=\frac12(y-1)(y+1)(2y+3),
\]
and, when $x=-1$,
\[
{\mathcal  A}_2=4y(y-2)^2,
 \qquad
 {\mathcal B}_2=\frac12(y-1)(y+1)(2y-3).
\]
In either case the two polynomials have no common real zero. Thus
also $r=2$ is not a possible value.
We conclude that
\begin{equation*}\label{r1}
 {r=1,\qquad q=3.}
\end{equation*}
Thus the principal curvature in the direction of $T$ is necessarily
simple.
We now determine the two principal curvatures. Since $x=\pm1$ is constant, the relation ${\mathcal P}_1(x,y)=0$ again implies that $y$ takes values in a finite set. Thus $y$ is constant on $U$, and \eqref{zdef} gives $z=0$. Substitution into ${\mathcal  A}_1={\mathcal B}_1=0$ gives
\[
(x,y)=(1,-2) \quad {\rm or} \quad (x,y)=(-1,2)\,.
\]
The two possibilities are exchanged by reversing the unit normal.
We therefore choose the orientation for which
\begin{equation*}\label{xyfinal}
 x=-1,\qquad y=2.
\end{equation*}
From \eqref{xydef}, we obtain
\[
 {
 \lambda=-v\quad\text{(multiplicity $3$)},\qquad
 \mu=2v\quad\text{(multiplicity $1$)}.
 }
\]
Substitution into \eqref{Fr} gives
\[
 {\mathcal F}_1(-1,2)
 =
 6\bigl(1+(-1)^2+(-1)\cdot2\bigr)=0.
\]
Hence
\begin{equation*}\label{scalminus12}
{\Scal_M=-12.}
\end{equation*}
Finally, since $T$ is principal and $u=\cos\alpha\neq0$, the hypersurface is
locally of Tojeiro type:
\[
 \phi(x,s)=\bigl(f_s(x),a(s)\bigr),
\]
where $f_s$ is a parallel family in $\h^4$. Put
\[
 \rho(s)=\frac{a'(s)}{\sqrt{1+a'(s)^2}},
\]
and recall that 
\begin{equation*}\label{vrho}
 {v=|T|=\rho.}
\end{equation*}
The principal curvatures of a Tojeiro hypersurface are
\[
 -\rho\,\kappa_1(s),\quad
 -\rho\,\kappa_2(s),\quad
 -\rho\,\kappa_3(s),\quad
 \rho'(s),
\]
where $\kappa_i(s)$ are the principal curvatures of
$f_s\subset\h^4$.
Since the $T$--principal curvature is simple and the other three
principal curvatures are all equal to $\lambda=-v=-\rho$, we have
\[
 -\rho\kappa_i=-\rho,\qquad i=1,2,3.
\]
As $\rho>0$,
\[
 \kappa_1=\kappa_2=\kappa_3=1.
\]
Therefore every $f_s$ is a totally umbilical hypersurface of
$\Hh^4$ with principal curvature $1$, hence a horosphere.
Moreover,
\[
 \rho'=\mu=2v=2\rho.
\]
Thus
\begin{equation*}\label{rhoexp}
{\rho(s)=ce^{2s}}
\end{equation*}
for some $c>0$. The parameter $c>0$ does not produce geometrically distinct hypersurfaces. Indeed, after introducing the translated parameter
\[
\widetilde{s}=s+\frac{\log c}{2},
\]
one has
\[
\rho(\widetilde{s})=e^{2\widetilde{s}}.
\]
Thus the dependence on $c$ can be absorbed by a translation of the parameter $s$ along the parallel family of horospheres.
\end{proof}

\begin{proof}[Proof of Theorem~\ref{Th-Max-Principle-application}] First, we state the classical Omori-Yau maximum principle in a form which is convenient for our purposes:
\begin{theorem}[\cite{MR215259,MR431040}, see also \cite{MR3445380}]\label{Th-Max-Princ} 

Assume that $M$ is a complete Riemannian manifold with Ricci curvature bounded from below. Let $u:M \to \R $ be a ${C}^2$ function. If $u$ is bounded from above,
then there exists a sequence of points $\{ p_k\}_{k \in \n} \subset M$ such that
\[
\Delta u(p_k) > -\,\frac{1}{k}\,, \quad \lim_{k \to \infty}|\grad u(p_k)|=0 \quad {\rm and}\quad \lim_{k \to \infty}u(p_k)= \sup_M u \,.
\]
\end{theorem}

\vspace{3mm}
The strategy of the proof is to apply the Omori-Yau maximum principle to the normal component of the bitension field.

First, using \eqref{BHP2}, we find:
\begin{eqnarray}\label{eq-DeltaH2<=0}\nonumber
\frac{1}{2}\Delta H^2&=&  H \Delta H -  |\grad H|^2 \\
&=&- [|A|^2 + (m-1)  \sin^2 \alpha \big ] H^2 -  |\grad H|^2 \leq 0 \,.
\end{eqnarray}
By assumption, $H^2$ is bounded from above. If ${\rm Ric}(M)$ is bounded from below, we can apply the maximum principle stated in Theorem~\ref{Th-Max-Princ} to the function $u=H^2$ and we deduce that there exists a sequence of points $\{ p_k\}_{k \in \n} \subset M^m$ such that
\[
\Delta H^2(p_k) > -\,\frac{1}{k} \quad {\rm and}\quad \lim_{k \to \infty}H^2(p_k)= \sup_M H^2 \,.
\]
Therefore, using \eqref{eq-DeltaH2<=0}, we deduce that 
\[
\lim_{k \to \infty}\Delta H^2(p_k) =0
\]
and so
\begin{equation}\label{eq:limite-utile}
\lim_{k \to \infty}- \left (\big [|A|^2 + (m-1)   \sin^2 \alpha \big ] H^2 \right )(p_k) =0\,.
\end{equation}
Next, we argue by contradiction and suppose that the hypersurface is not minimal. Thus
\begin{equation*}\label{eq-sup>0}
\sup_M H^2 >0\,.
\end{equation*}
Then, from \eqref{eq:limite-utile}, we infer
\begin{equation}\label{eq-max-princ-hyper}
\lim_{k \to \infty} |A|^2(p_k) =0\,.
\end{equation}

Next, we recall the general fact
\[
(m-1) H^2 \leq |A|^2 \,.
\]
This inequality, together with \eqref{eq-max-princ-hyper}, imply
\begin{equation*}\label{eq-1-limite}
\lim_{k \to \infty} H^2(p_k)=0\,,
\end{equation*}
a contradiction, as required to end the proof.
\end{proof} 
\begin{remark}
The assumption on the Ricci curvature in Theorem~\ref{Th-Max-Princ} can be relaxed. Indeed, the validity of the Omori-Yau maximum principle can be ensured by either some function theoretic assumptions or by assuming a weaker lower bound on the Ricci curvature. We refer to \cite{MR3445380, MR2116555} for a thorough discussion on this topic. More precisely (see \cite[Theorem~2.5]{MR3445380}), in Theorem~\ref{Th-Max-Princ} and Theorem~\ref{Th-Max-Principle-application}, it suffices to assume
\begin{equation*}\label{eq-Ricci-bound}
{\rm Ric}(M)_x \geq - (m-1) G^2(r(x))\langle\,, \rangle_M \,,
\end{equation*}
where $r(x)$ denotes the distance function from a fixed point and $G$ is a $C^1$ function on $[0,+\infty)$ such that
\[
{\rm (i)} \, G(0)>0\, \quad {\rm (ii)}\, G'(t) \geq 0\, \quad {\rm (iii)}\, \frac{1}{G(t)} \not \in {\rm L}^1 ([0,+\infty))\,.
\]
\end{remark}
\begin{remark} The method of proof of Theorem~\ref{Th-Max-Principle-application} does not naturally apply to the case of complete biharmonic hypersurfaces in $\s^m \times \R$. Nevertheless, we can prove a related result which requires different, somehow stronger assumptions. More specifically, we have:
{
\begin{proposition}\label{Prop-max-princ-sphere}
Let $\phi:M^{m}\hookrightarrow \s^m \times \R$ be a complete biharmonic hypersurface. Assume that
\[
 |A|^2 \geq m-1 
 \]
 and $H^2$ is of class $L^p$ for some $p>1$.
Then  $\phi:M^{m}\hookrightarrow \s^m \times \R$ is minimal.\end{proposition} 
\begin{proof} Using \eqref{BHP2} and the assumption $|A|^2 \geq m-1$ we have:
\begin{eqnarray*}
\frac{1}{2}\Delta H^2&=&  H \Delta H -  |\grad H|^2 \\
&=&- \big [|A|^2 - (m-1)  \sin^2 \alpha \big ] H^2 -  |\grad H|^2 \leq 0 \,.
\end{eqnarray*}
Now, since $H^2$ is of class $L^p$ for some $p>1$, a classical result of Yau \cite[Theorem 3]{MR417452} (see also \cite{MR662919}) implies that $H$ is constant. Assume that $H\neq 0$, then,  from Proposition~\ref{Prop-CMC} (ii),  $\phi:M^{m}\hookrightarrow \s^m \times \R$ is a cylinder $M^{m-1}\times\R$, where $M^{m-1}$ is a non-minimal CMC biharmonic hypersurface in $\s^m$. Now, in the cylinder $M^{m-1}\times\R$, $H\neq 0$ contradicts that $H^2$ is of class $L^p$ for some $p>1$.
\end{proof}
}

Finally, we point out that, according to the Proposition~4.1 of \cite{MR3107694}, a {\em compact} biharmonic hypersurface in $\s^m\times\R$ is a slice and, by the same argument, we conclude that there exists no compact biharmonic hypersurface in $\h^m\times\R$.
\end{remark}

\end{document}